\documentclass{amsart}

\usepackage[T1]{fontenc}
\usepackage{times}
\usepackage{amssymb,epsfig,verbatim,xypic}

\theoremstyle{plain}

\usepackage[T1]{fontenc}
\usepackage[utf8]{inputenc}
\usepackage[english]{babel}
\usepackage{amssymb}
\usepackage{amsthm}
\usepackage{graphics}
\usepackage{amsmath}
\usepackage{amstext}
\usepackage{multirow}

\usepackage{graphicx}
\usepackage{amsmath}
\usepackage{amssymb}
\usepackage{amsthm}
\usepackage{amstext}
\usepackage{epstopdf}
\usepackage{mathrsfs}
\usepackage{amscd}
\usepackage{tikz}
\usetikzlibrary{cd}

\AtBeginDocument{%
   \def\MR#1{}
}

\newtheorem{thm}{Theorem}[section]
\newtheorem{dfn}[thm]{Definition}
\newtheorem{lemma}[thm]{Lemma}
\newtheorem{prop}[thm]{Proposition}
\newtheorem{cor}[thm]{Corollary}
\newtheorem{question}[thm]{Question}
\newtheorem{problem}[thm]{Problem}

\newtheorem{THM}{Theorem}

\theoremstyle{remark}
\newtheorem{remark}[thm]{Remark}

\newcommand{\mb}{\mathbb}
\newcommand{\tor}{\mathrm{tor}}

\newcommand{\C}{\mb C}

\DeclareMathOperator{\Pic}{Pic}

\DeclareMathOperator{\Aff}{Aff}

\DeclareMathOperator{\Res}{Res}
\DeclareMathOperator{\Spec}{Spec}
\DeclareMathOperator{\utype}{utype}
\DeclareMathOperator{\uclass}{uclass}

\DeclareMathOperator{\alg}{alg}

\numberwithin{equation}{section}

\addtocounter{section}{0}             
\numberwithin{equation}{section}       

\title{On the formal principle for curves on projective surfaces}
\author{Jorge Vit\'orio Pereira}
\author{Olivier Thom}
\date{\today}

\begin{document}

\begin{abstract}
We prove that the formal completion of a complex projective surface along a  rigid
smooth curve with trivial normal bundle determines the birational equivalence class of
the surface.
\end{abstract}

\maketitle
\setcounter{tocdepth}{1}

\section{Introduction}

In this paper, we investigate pairs $(X,Y)$ of complex varieties where
$Y$ is a compact subvariety of the complex variety $X$. We are particularly
interested in the analytic classification of such pairs when $X$ is a smooth projective
surface.

\begin{dfn}
    A pair $(X,Y)$ satisfies the {\bf formal principle} if for any other pair $(X',Y')$ such that
    the formal completion $\mathscr Y$ of $X$ along $Y$ is formally isomorphic to the formal
    completion $\mathscr Y'$ of $X'$ along $Y'$ then the germ of $X$ along $Y$ is biholomorphic
    to the germ of $X'$ along $Y'$.
\end{dfn}

If $Y$ is a smooth compact curve on a smooth surface $X$ with non-zero self-intersection then
the pair $(X,Y)$ satisfies the formal principle. Indeed, if $Y^2<0$ then  \cite[Section 4, Satz 6]{Grauert} implies that $(X,Y)$ satisfies the formal principle. When $Y^2>0$, then the result is implied by \cite{GrauertCommichau}, see also the discussion in \cite[Section 4]{Kosarew}. The case of zero self-intersection is in sharp contrast. To the best of our knowledge, the first example of a pair $(X,Y)$ for which the formal principle does not hold is due to  V.I. Arnold: it consists of a germ of surface $X$ containing an elliptic curve of zero self-intersection and non-torsion normal bundle obtained through the suspension of a germ of non-linearizable  biholomorphism, see \cite{Arnold}. Even if one restricts to neighborhoods of elliptic curves with trivial normal bundles,  the analytic classification differs considerably from the formal classification, see  \cite[Theorem 5]{LorayThomTouzet}.

There are many more works investigating the formal principle. We invite the reader to consult the recent paper \cite{MR3968811} and the surveys in \cite{Kosarew} and  \cite[Section VII.4]{SCVVII} to get a  view of different directions of research on the subject.

\subsection{Projective formal principle} The results just mentioned provide an abundance of pairs for
which the formal principle does not hold. They are based on local analytic construction and do not globalize. Indeed, Neeman in \cite[Article 1, Theorem 6.12]{Neeman} shows that a smooth elliptic curve $Y$ with trivial normal bundle  on a projective surface $X$ either is a  fiber of a fibration, or $X$ is birationally equivalent to $\mathbb P(E)$, the projectivization of the unique rank two vector bundle over $Y$ obtained as a non-trivial extension of the trivial line-bundle by itself, and $Y$ corresponds to the natural section $Y \to \mathbb P(E)$.

Taking into account this result, it seems natural to consider the following restricted version of the formal principle.

\begin{dfn}
    A pair $(X,Y)$ satisfies the {\bf projective formal principle} if $X$ is a projective variety
    and for any other pair $(X',Y')$ such that $X'$ is projective and the formal completion $\mathscr Y$
    of $X$ along $Y$ is formally isomorphic to the formal completion $\mathscr Y'$ of $X'$ along $Y'$
    then the germ of $X$ along $Y$ is biholomorphic to the germ of $X'$ along $Y'$. Furthermore, we
    will say that $(X,Y)$ satisfies the {\bf birational formal principle} when, under the same assumptions
    as above, there exists a birational map between $X$ and $X'$ which sends, biregularly, a neighborhood of  $Y$ to a neighborhood of $Y'$.
\end{dfn}

\subsection{Smooth curves on projective surfaces}
Our first main result says that the projective formal principle holds for smooth curves with trivial normal bundle on projective surfaces.

\begin{THM}\label{THM:A}
    Let $(S,C)$ be a pair where $S$ is a smooth projective surface and $C$ is a smooth curve contained in $S$. If the normal bundle of $C$ in $S$ is trivial then $(S,C)$ satisfies the projective formal principle. Moreover, if $C$ is not a fiber of a fibration in $S$ then $(S,C)$ satisfies the birational formal principle.
\end{THM}

When the curve $C$ is a fiber of a fibration on a smooth surface $S$ (projective or just a germ of), a result by Hirschowitz \cite{Hirschowitz}, see also \cite[Theorem 2.2]{Kosarew}, says that the pair $(S,C)$ satisfies the formal principle. Theorem \ref{THM:A} adds nothing to this statement. The real content of it is when $C$ is not a fiber of a fibration. Its proof is built on  the existence of
natural closed rational $1$-forms with polar set equal to $C$ \cite[Theorem B]{CLPT}(a result which uses basic Hodge Theory for compact Kähler manifolds) to guarantee the convergence of any given formal isomorphism, and we use a result by Ueda to extend the (now convergent) isomorphism to birational  maps.

\subsection{On the Ueda type of hypersurfaces on projective manifolds}
Our second main result is a common generalization of \cite[Theorem B]{CLPT}, \cite[Theorem 5.6]{Neeman}, and \cite[Theorem 2.3]{LPTFlat}.
Although \cite[Theorem B]{CLPT} is sufficient to prove Theorem \ref{THM:A}, the result below seems to be of independent interest,
and it might prove to be useful to investigate the projective formal principle in different situations.

\begin{THM}\label{THM:closed}
    Let $D$ be a connected divisor on a compact Kähler manifold $X$. Assume the existence
    of a line bundle $\mathcal L \in \Pic^{\mathrm{tor}}(X)$ and of a positive integer $k$ such that
    $\mathcal O_X(kD)_{||D|} \simeq \mathcal L_{||D|}$. Then the following assertions hold true:
    \begin{enumerate}
        \item If $\mathcal O(kD)_{|kD + D_{red} }  \simeq \mathcal L_{|kD + D_{red}}$ then, perhaps after replacing
        $X$ by a degree two  étale covering, there exists a global closed logarithmic $1$-form $\omega$ with purely imaginary periods such that
        \[
            \Res \omega_{|U} = D
        \]
        where $U$ is a  sufficiently small neighborhood of $|D|$.
        \item If $\mathcal O(kD)_{|kD}  \simeq \mathcal L_{|kD}$ then there exists a global closed meromorphic  $1$-form $\omega$
        with coefficients in $\mathcal L^*$, without residues, and with polar divisor equal to $kD + D_{red}$.
    \end{enumerate}
\end{THM}

In the statement of Theorem \ref{THM:closed}, $\Pic^{\mathrm{tor}}(X)$ denotes the group
of isomorphism classes of line-bundles on $X$ with torsion Chern class. Also, in Item (1) of the statement above, the periods of $\omega$ are the integrals of $\omega$ along closed real curves not intersecting the polar locus of $\omega$. Throughout the paper, starting in the statement above, we will say by abuse of language
that a meromorphic form on a projective manifold with coefficients in a line-bundle $\mathcal L\in \Pic^{\mathrm{tor}}(X)$
is closed, if it is $\nabla$-closed for  the unique flat unitary connection $\nabla$ on $\mathcal L$.

\subsection{Acknowledgements} J.V. Pereira thanks Jun-Muk Hwang for calling his attention
to the literature on the formal principle. Both J.V. Pereira and O. Thom are grateful to Frank Loray for helpful discussions and to the anonymous referee for the careful reading and thoughtful suggestions. J. V. Pereira was supported by Cnpq and FAPERJ. O. Thom was  supported by Cnpq.
Both authors acknowledge support from CAPES-COFECUB Ma 932/19 project.

\section{Ueda theory}

In this section, we recall the definitions of Ueda type and Ueda class of smooth divisors with
topologically trivial normal bundle. We adopt the point of view presented by Neeman in \cite{Neeman},
and try to be coherent with notations used in \cite{CLPT}.

\subsection{Ueda line bundle}
Let $Y$ be a smooth irreducible compact hypersurface on a complex manifold $U$.
Assume that the normal bundle $N_Y$ is topologically torsion and that $Y$ and $U$ share the same homotopy type.
Let us assume the existence of a flat unitary connection on $N_Y$. This condition is automatically fulfilled if
$Y$ is Kähler. The monodromy representation $\rho_Y : \pi_1(Y) \to S^1 \subset \mathbb C^*$  of this unitary connection
induces a representation of $\pi_1(U)$ into $S^1$ since we are assuming that $Y$ and $U$ have the same homotopy
type. We will denote by $\widetilde{N_Y}$ the (flat unitary) line bundle on $U$ determined by the extension
of $\rho_Y$ to $\pi_1(U)$. Notice that $\widetilde{N_Y}$ extends $N_Y$ in the sense that $\widetilde{N_Y}_{|Y} = N_Y$.

The line-bundle $\mathcal O_U(Y)$ is another extension of $N_Y$ to $U$. The {\bf Ueda line bundle} is, by definition, the
line-bundle $\mathcal U = \mathcal O_U(Y) \otimes \widetilde{N_Y}^*$.

\subsection{Ueda type and Ueda class}
Let $I\subset \mathcal O_U$ be the ideal sheaf  defining $Y$.
We will denote  the $k$-th infinitesimal neighborhood of $Y$ in $U$ by $Y(k)$, \emph{i.e.}
\[
    Y(k) = \Spec ( \mathcal O_U / I^{k+1} ).
\]
The {\bf Ueda type}  of $Y$ ($\utype (Y)$) is equal to $\infty$ if $\mathcal U_{|Y(\ell)} \simeq \mathcal O_{Y(\ell)}$ for every non-negative integer $\ell$, otherwise
is the smallest positive integer $k$ such that $\mathcal U_{|Y(k)} \not\simeq \mathcal O_{Y(k)}$. In other words, the Ueda type is infinite if and only if the restriction
of $\mathcal U$ to $\mathscr Y$, the formal completion of $U$ along $Y$, is trivial.

If $\utype(Y) = k < \infty$ then the {\bf Ueda class} of $Y$ is defined as the element in the cohomology group $H^1(Y, I^{k}/I^{k+1})$ mapped  to $\mathcal U_{|Y(k)} \in \Pic(Y(k))$
through the truncated exponential sequence
\begin{center}
\begin{tikzcd}
 H^1(Y, I^{k}/I^{k+1}) \arrow[r, "\exp" ] &  \Pic (Y(k))\arrow[rr,"\rm{restriction}"] && \Pic(Y(k-1)) \, .
\end{tikzcd}
\end{center}
Indeed, under the assumption that $\mathcal U_{|Y(k-1)} \in \Pic(Y(k-1))$ is trivial, Neeman shows that map on the left is injective \cite[Remark 1.7]{Neeman} and one gets a well-defined class in $H^1(Y, I^{k}/I^{k+1})$.

For a more concrete interpretation of the Ueda type and Ueda class in terms of defining equations for $Y$ and the associated cocycles, we
invite the reader to consult Ueda's original definition in  \cite{Ueda} and the discussion carried out in \cite[Section 2.1]{CLPT}. Here
we point out the following characterization of neighborhoods where $\utype(Y)= \infty$.

\begin{lemma}\label{L:log unitary}
    Notation as above. The Ueda type of $Y$ is infinite if, and only if, there exists a (formal) closed logarithmic
    $1$-form $\omega \in H^0(\mathscr Y, \Omega^1_{\mathscr Y}(\log Y))$ with $\Res(\omega) = Y$ and purely imaginary
    periods.
\end{lemma}
\begin{proof}
    If $\mathcal U$ is trivial  then there exists a covering $\{ U_i\}$ of $\mathscr Y$ and (formal) functions $f_i \in \mathcal O_{\mathscr Y}(U_i)$
    such that $Y \cap U_i = \{ f_i = 0 \}$ and $f_i = \lambda_{ij} \cdot f_j$ over $U_i \cap U_j$. The sought $1$-form $\omega$ is then defined  over $U_i$
    as the logarithmic derivative of $f_i$. Clearly,  $\Res(\omega)= Y$ and it has purely imaginary periods.

    Reciprocally, suppose there exists $\omega$ with purely imaginary periods and   with $\Res(\omega)= Y$. Then, over a simply connected open
    covering $\{ U_i\}$ of $\mathscr Y$, we can set
    \[
        f_i = \exp \left( \int \omega_{|U_i} \right) \, .
    \]
    Since the periods of $\omega$ are purely imaginary, over $U_i \cap U_j$ the quotient $f_i/f_j$ is a complex number of modulus one. This
    shows that $\mathcal O_U(Y)_{|\mathscr Y}$ is isomorphic to ${\widetilde{N_Y}}_{|\mathscr Y}$. The triviality of $\mathcal U_{|\mathscr Y}$ follows.
\end{proof}

\subsection{Hypersurfaces of infinite type}

By definition, if $\utype(Y)=\infty$, then the restriction of $\mathcal U$ to the completion of $U$ along $Y$ is trivial, \emph{i.e.} $\mathcal U$ is trivial on a formal neighborhood of $Y$.
The theorem below, due to Ueda (cf. \cite[Theorem 3]{Ueda}), gives sufficient conditions to the triviality of $\mathcal U$ on an Euclidean neighborhood of $Y$ in $U$.  Although Ueda states his result only for curves on surfaces, his proof works as it is to establish the more general  result below.

\begin{thm}\label{T:Uedainfty}
Let $Y$ be a smooth compact  connected K\"ahler hypersurface of a complex manifold $U$ with topologically torsion normal bundle.
If $\utype(Y)= \infty$ and $N_Y$ is a torsion line-bundle then $Y$ is a (multiple) fiber of a fibration.
\end{thm}

\subsection{Hypersurfaces of finite type and closed formal differential forms}
As above, let $\mathscr Y$ be the formal completion of $U$ along $Y$.

\begin{lemma}\label{L:holomorfa}
    If $\utype(Y)<\infty$ then every formal holomorphic function on $\mathscr Y$ is constant.
\end{lemma}
\begin{proof}
    Aiming at a contradiction, assume the existence of a non-constant formal function $f \in H^0(\mathscr Y, \mathcal O_{\mathscr Y})$. The function $f$ is necessarily constant along the hypersurface $Y$. The formal function $g = f - f(Y)$ provides a (non necessarily reduced) equation for $Y$. The logarithmic differential of  $g$ is a closed logarithmic $1$-form with residue divisor $Y$ and purely imaginary periods. Lemma \ref{L:log unitary} gives the sought contradiction.
\end{proof}

\begin{lemma}\label{L:restriction 1}
    If $\utype(Y) < \infty$ then the restriction morphism
    \[
        \ker \left\{ d: H^0(\mathscr Y, \Omega^1_{\mathscr Y}) \to H^0(\mathscr Y, \Omega^2_{\mathscr Y}) \right\} \longrightarrow H^0(Y,\Omega^1_Y)
    \]
    from closed formal holomorphic $1$-forms on $\mathscr Y$ to $H^0(Y,\Omega^1_Y)$ is injective.
\end{lemma}
\begin{proof}
    A closed formal $1$-form is locally the differential of a formal holomorphic function. If $\omega \neq 0$ is a closed formal $1$-form which vanishes when restricted to $Y$, i.e. is in the kernel of the morphism above, then its local primitives are locally constant along $Y$. Thus, if we choose such primitives vanishing along $Y$ then they patch together to give a unique formal holomorphic function $f$ vanishing along $Y$ and such that $\omega = df$. Lemma \ref{L:holomorfa} implies the result.
\end{proof}

\subsection{Meromorphic functions on neighborhoods of curves of finite Ueda type}
A surprising phenomenon, discovered by Ueda, is that the complex geometry of neighborhoods of curves of finite
Ueda type shares features with the complex geometry of neighborhoods of ample curves, see \cite[Corollary of Theorem 1]{Ueda}.

\begin{thm}\label{T:Ueda}
    Let $C$ be a  smooth curve on a smooth  surface $U$. If  $C^2 = 0$ and $\utype(C) < \infty$, then $C$
    has a fundamental system of strictly pseudoconcave neighborhoods.
\end{thm}

Although the field of formal meromorphic functions on the completion of $U$ along $C$ is of infinite transcendence
degree over the complex numbers (\cite[Section 5]{Hironaka-Matsumura}), Theorem \ref{T:Ueda} combined with a result by Andreotti
\cite[Theorem 4]{Andreotti} guarantee the oppositive behavior for the field of germs of meromorphic functions
on neighborhoods of $C$.

\begin{cor}\label{C:transcendence}
    Let $C$ and $U$ be as in Theorem \ref{T:Ueda}.
    The transcendence degree of the field of germs of
    (convergent) meromorphic functions on neighborhoods of $C$ is at most two.
\end{cor}

Theorem \ref{T:Ueda} has strong consequences when $C$ is a curve in a smooth projective surface. We collect
some of them in the statement below.

\begin{cor}\label{C:UedaProjective}
    Let $C$ be a  smooth curve on a smooth  projective surface $S$. If  $C^2 = 0$ and $\utype(C) < \infty$, then
    the following assertions hold true:
    \begin{enumerate}
        \item\label{I:Narasimhan} $S - C$ is holomorphically convex and after the contraction of finitely many curves it becomes a Stein space;
        \item\label{I:H1} The morphism $\iota_*:H_1(C,\mathbb Z) \to H_1(S,\mathbb Z)$ induced by the inclusion of $\iota : C \to S$ is surjective;
        \item\label{I:sections} Every meromorphic function defined on an Euclidean neighborhood of $C$ extends to a global rational function. More generally,
        if $\mathcal E$ is a locally free sheaf of $\mathcal O_S$-modules then every meromorphic section of $\mathcal E$ defined on
        an Euclidean neighborhood of $C$ extends to a global rational section.
        \item\label{I:forms} If $\omega$ is a closed holomorphic $1$-form defined on a neighborhood $C$ then $\omega$ extends to
        a global holomorphic $1$-form defined on $S$.
    \end{enumerate}
\end{cor}
\begin{proof}
    Theorem \ref{T:Ueda} implies that $S-C$ satisfies the assumptions of  \cite[Theorem 1]{Narasimhan}, Item \ref{I:Narasimhan} follows.
    Item \ref{I:H1} is the content of \cite[Lemma 4]{UedaSingular} and Item \ref{I:sections} is the content of \cite[Lemma 5]{UedaSingular}.
    Finally, Item \ref{I:forms} is the content of  \cite[Theorem 3]{UedaSingular}.
\end{proof}

\section{Existence of closed rational $1$-forms}\label{S:closed}

This section is devoted to the proof of Theorem \ref{THM:closed} from the introduction.

\subsection{Algebraic fundamental group} We start things off with a generalization of Item (\ref{I:H1}) of Corollary \ref{C:UedaProjective}.

\begin{prop}\label{P:pi1}
    Let $D$ be an effective compact and connected divisor on a compact Kähler manifold $X$ with $D^2=0$ in $H^4(X, \mathbb Q)$. Assume that the natural morphism
    \[
        \iota_* : \pi_1^{\mathrm{alg}}(|D|) \longrightarrow \pi_1^{\mathrm{alg}}(X)
    \]
    induced by the inclusion of $\iota:|D| \to X$ of the support of $D$ into $X$  is not surjective.
    Then, perhaps after replacing $X$ by a degree two étale covering, there exists a global closed logarithmic $1$-form $\omega$ with purely imaginary periods such that
    \[
        \Res \omega_{|U} = D
    \]
    where $U$ is a  sufficiently small neighborhood of $|D|$.
\end{prop}
\begin{proof}
    If $\iota_*$ is not surjective then, by definition, there exists a finite group $\Gamma$ and a surjective morphism
    $\rho : \pi_1(X) \to \Gamma$ such that $k = [\Gamma:\rho(\iota_* \pi_1(|D|))] > 1$.

    Let $r: Y\to X$ be the  Galois covering determined by the kernel of $\rho$.  It is a finite étale covering of degree equal to the cardinality of $\Gamma$. Since we are assuming that $X$ is compact Kähler, the same holds true for $Y$.

    Consider the divisor $r^*D$. It is a compact divisor on $Y$. The action of $\Gamma$ on $Y$ preserves  $r^*D$ and therefore acts on the set of its connected components.  Since the subgroup $\rho(\iota_* \pi_1(|D|))\subset \Gamma$ can be identified  with the subgroup which acts trivially on the
    set of connected components of  $r^*D$, it follows that  $r^*D$ has exactly $k>1$ connected     components.

    Assume first $k=2$. In this case,   $r^*D= D_1 + D_2$ where $D_1$ and $D_2$ are connected disjoint divisors with $D_1^2 = D_2^2 = 0$.  Let $\Theta$ be a Kähler form on $Y$, and consider the bilinear form on $H^{1,1}(Y)$ defined by $\alpha \cdot \beta = \int_Y \alpha \wedge \beta \wedge \Theta^{\dim X - 2}$ for any $\alpha, \beta \in H^{1,1}(Y)$. Hodge index theorem (see for instance \cite[Section 6.3.2]{Voisin})  implies this bilinear form has signature $(1,h^{1,1}(Y)-1)$. Since $D_1^2=D_1\cdot D_2=D_2^2=0$, it follows that a multiple of $m_1D_1$ of $D_1$  is numerically equivalent to a multiple $m_2 D_2$ of $D_2$. Therefore the line bundle $\mathcal O_{Y}(m_1 D_1 - m_2 D_2)$ has trivial Chern class, and  as such supports a flat unitary connection. As explained in \cite[Proposition 3.2]{PereiraJAG}, this connection uniquely determines a closed logarithmic $1$-form $\omega$ with purely imaginary periods and $\Res(\omega) = m_1 D_1 - m_2 D_2$. The sought $1$-form is  $(1/m_1)\omega$.

    If $k \ge 3$  the existence of a fibration $f: \overline Y \to C$ mapping the connected components of $|r^*D|$ to distinct points follows from \cite[Theorem 2.1]{Totaro} (although stated only for projective manifolds,  the result is also true for compact Kähler manifolds \cite[Theorem 2]{PereiraJAG}).
    The foliation defined by this fibration is unique, and as such, must be preserved by the action of $\Gamma$ on $Y$. It follows the existence of a fibration on $X$ with $|D|$ equal to the support of one of its fibers. The existence of $\omega$ follows easily by pulling back a suitable logarithmic $1$-form on the basis of the fibration.
\end{proof}

\begin{cor}\label{C:pi1}
    Let $X$ be a compact Kähler manifold and let $Y$ be a smooth compact hypersurface of $X$ with numerically trivial normal bundle. If $\utype(Y) < \infty$ then the natural morphism
    \[
        \pi_1^{\mathrm{alg}}(Y) \longrightarrow \pi_1^{\mathrm{alg}}(X)
    \]
    is surjective.
\end{cor}
\begin{proof}
    Let us prove the contrapositive assertion. For that assume that the morphism in question is not surjective. Proposition \ref{P:pi1} implies the existence of $1$-form $\omega$ with purely imaginary periods and $\Res(\omega) = Y$. Lemma \ref{L:log unitary} implies  $\utype(Y) = \infty$ as wanted.
\end{proof}

\subsection{Hodge theory for unitary flat line-bundles}
Let $X$ be a compact Kähler manifold and $\mathbb L$ be a rank one local system with unitary monodromy. The line-bundle $\mathcal L = \mathbb L \otimes_{\mathbb C} \mathcal O_X$ comes endowed with a canonical flat unitary holomorphic connection $\nabla : \mathcal L \to \mathcal L \otimes \Omega_X^1$ characterized by the property that the local system of flat sections of $\nabla$ is  equal to $\mathbb L \otimes 1 \subset \mathcal L$.

Harmonic theory on compact Kähler manifolds adapts to the study of harmonic forms with coefficients on $\mathcal L$. In particular, the following consequence of the $\partial \overline{ \partial }$-lemma also holds in this more general context, see for instance \cite[(3.3)]{Beauville}.

\begin{lemma}\label{L:Beauville}
    The homomorphisms
    \[
	H^q(\nabla) : H^q(X,\Omega^{p}_X \otimes \mathcal L) \to H^{q}(X, \Omega^{p+1}_X \otimes \mathcal L)
    \]
    are zero.
\end{lemma}

The twisted De Rham complex $(\Omega^{\bullet}_X \otimes \mathcal L, \nabla)$ provides a resolution of $\mathbb L$, and therefore one deduces from Lemma \ref{L:Beauville}, as in the case of constant coefficients, a Hodge decomposition
\[
    H^i(X,\mathbb L) = \bigoplus_{p+q=i} H^p(X, \Omega^q_X \otimes \mathcal L) \, .
\]
Furthermore, complex conjugation of harmonic forms yields (sesqui-linear) isomorphisms
\[
    H^q(X, \Omega^p_X \otimes \mathcal L) \to H^p(X, \Omega^q_X \otimes \mathcal L^*) .
\]

We point out also that if $\mathcal L$ is a line-bundle with trivial Chern class on a compact Kähler manifold $X$
and we consider the unique  flat unitary connection $\nabla$ on it, then any global section $\omega$ of $H^0(X, \Omega^q_X \otimes \mathcal L)$
is automatically $\nabla$-closed since Stoke's Theorem implies
\[
    \int_X \nabla(\omega) \wedge \overline{\nabla(\omega)} \wedge \Theta^{\dim X-(q+1)} = 0
\]
for any Kähler form $\Theta$.

\subsection{Restriction of cohomology classes}

\begin{prop}\label{P:restriction}
    Let $D$ be an effective connected divisor on a compact Kähler manifold $X$ with $D^2 = 0$ in $H^4(X,\mathbb Q)$. Let $(\mathcal L,\nabla)$ be a line bundle endowed with a flat unitary connection on $X$. If the restriction morphism
    \[
        H^1(X, \mathcal L) \longrightarrow H^1(|D|, \mathcal L)
    \]
    is not injective then, perhaps after replacing $X$ by a degree two étale covering, there exists a global closed logarithmic $1$-form $\omega$ with purely imaginary periods such that
    \[
        \Res \omega_{|U} = D
    \]
    where $U$ is a  sufficiently small neighborhood of $|D|$.
\end{prop}
\begin{proof}
    Let $\mathbb L$ be the local system of flat sections of $\mathcal L$.
    Functoriality of Hodge decomposition implies that the restriction morphism
    \[
            H^1(X, \mathbb L) \longrightarrow H^1(|D|, \mathbb L) \,
    \]
    is also not injective. Let $\alpha$ be a non-zero element in its kernel.

    If $\rho : \pi_1(X) \to \mathbb C^*$ is the monodromy representation of the
    local system $\mathbb L$ then any cohomology class $0 \neq \alpha \in H^1(X, \mathbb L)$ corresponds to a non-abelian representation $\hat \rho : \pi_1(X) \to \Aff(\C)$ with linear part given by $\rho$.

    If $\alpha$ restricts to zero in $H^1(|D|, \mathbb L)$ then the composition
    \begin{center}
    \begin{tikzcd}
           \pi_1(|D|) \arrow[r, "\iota_*" ] \arrow[rr, bend left=20] &  \pi_1(X) \arrow[r, "\widehat{\rho}"] &  \Aff(\C) \,
    \end{tikzcd}
    \end{center}
    has abelian image.

    Since $\Aff(\C)$ is a linear algebraic group and finitely generated subgroups of linear algebraic groups are residually finite (Malcev's Theorem), the morphism $\widehat{\rho}$ factors through the canonical morphism $\pi_1(X) \to \pi_1^{\alg}(X)$. Consequently, the induced composition
    $\pi_1^{\alg}(|D|) \to \pi_1^{\alg}(X) \to  \Aff(\C)$
    also has abelian image, while the image of the rightmost arrow coincides with the image of $\widehat{\rho}$ and, in particular, is non-abelian.
    Thus the natural morphism
    \[
        \pi_1^{\alg}(|D|) \longrightarrow \pi_1^{\alg}(X)
    \]
    is not surjective. We apply Proposition \ref{P:pi1} to conclude.
\end{proof}

\subsection{Proof of Theorem \ref{THM:closed}}
    Since $\mathcal L$ is a line bundle with zero first Chern class, there exists a unique unitary flat connection $\nabla : \mathcal L \to \mathcal L \otimes \Omega^1_X$ on $\mathcal L$. Let $\mathbb L$ be the local system of flat sections of $\mathcal L$ with respect to $\nabla$.

    Assume $\mathcal O(kD)_{|kD}  \simeq \mathcal L_{|kD}$. In terms of a sufficiently fine open covering $\{U_i\}$ of $X$, this implies the existence of holomorphic functions $f_i \in \mathcal O_X(U_i)$ defining $D_{|U_i}$, complex numbers $\lambda_{ij}$  of modulus $1$, and holomorphic functions $r_{ij} \in \mathcal O_X(U_i\cap U_j)$  such that
    \[
        f_i^k = \left(\lambda_{ij} + f_j^k \cdot r_{ij}\right) \cdot f_j^k
    \]
    over the non-empty intersections $U_i \cap U_j$. This identity implies that the functions $a_{ij}$ defined by the formula
    \begin{equation}\label{E:formula}
        a_{ij} = \frac{1}{f_i^k} - \frac{1}{\lambda_{ij}} \cdot \frac{1}{f_j^k}
    \end{equation}
    are holomorphic on $U_i \cap U_j$. From its definition, its is clear that the collection
    $\{a_{ij} \}$ determines an element of $H^1(X, \mathcal L^*)$.

    Lemma \ref{L:Beauville} implies that class of $\{ da_{ij} \}$ in $H^1(X, \Omega^1_X \otimes \mathcal L^*)$ is zero. Therefore, perhaps after refining the open covering $\{U_i\}$, we can write over $U_i\cap U_j$.
    \[
        d a_{ij} = \alpha_i - \frac{1}{\lambda_{ij}} \cdot \alpha_j
    \]
    for suitable holomorphic $1$-forms $\alpha_i \in \Omega^1_X(U_i)$.

    Therefore the $1$-forms
    \[
       \omega_i =  d \left( \frac{1}{f_i^k} \right) - \alpha_i
    \]
    satisfy
    \[
        \omega_i = \frac{1}{\lambda_{ij}} \cdot \omega_j \,
    \]
    and hence define a global  rational $1$-form $\omega$ with coefficients in $\mathcal L^*$.

    It remains to verify that $\omega$ is closed. For that, let $\Theta$ be a Kähler form and  observe that $d \omega$ is clearly holomorphic. Stoke's Theorem implies
    \[
        \int_X  d \omega \wedge \overline{d \omega} \wedge \Theta^{\dim X-2} = \lim_{\varepsilon \to 0} \int_{\partial T_{\varepsilon}} \omega \wedge \overline{d\omega} \wedge \Theta^{\dim X-2}
    \]
    where  $T_{\varepsilon}$ is an $\varepsilon$-small tubular neighborhood of $|D|$. Since the right hand side is clearly equal to
    zero, it follows that $\omega$ is closed.

    If we further assume that  $\mathcal O(kD)_{|kD + D_{red}}  \simeq \mathcal L_{|kD + D_{red}}$ then the     restriction of $\{ a_{ij} \}$ to the support of $D$ is zero. We have two possibilities: the cohomology class determined by  $\{a_{ij}\}$   in $H^1(X,\mathcal L^*)$ is zero, or not. If it  is zero then we can assume that $\{a_{ij}\} = 0$. We  construct the sought  logarithmic $1$-form  by taking the logarithmic differential of Equation (\ref{E:formula}).
    If instead, the class of $\{a_{ij}\}$ is not zero in $H^1(X,\mathcal L^*)$ then we deduce that the restriction morphism $H^1(X,\mathcal L^*) \to H^1(|D|, \mathcal L^*)$     is not injective. We apply Proposition \ref{P:restriction} to conclude.
\qed

\subsection{Bound on the Ueda type of hypersurfaces}
Let $Y$ be a smooth hypersurface of a projective manifold $X$. Define
$\Pic^{\tor}(Y/X)$ as the cokernel of the restriction morphism
\[
    \Pic^{\tor}(X) \longrightarrow \Pic^{\tor}(Y) \, .
\]

Theorem \ref{THM:closed} admits the following immediate consequence.

\begin{cor}
Let $Y$ be a smooth hypersurface of a compact Kähler manifold $X$. Assume $N_Y$ is numerically trivial and that the order of the normal bundle of $Y$ in $\Pic^{\tor}(Y/X)$ is a finite integer $k$. Then $\utype(Y) \le k$ or $\utype(Y) = \infty$.
\end{cor}
\begin{proof}
    Since, by assumption, the order of the normal bundle of $Y$ in $\Pic^{\tor}(Y/X)$ is equal to $k$, there exists
    a line-bundle $\mathcal L  \in \Pic^{\tor}(X)$ such that $N_Y^{\otimes k} = \mathcal L$, or equivalently, $\mathcal O_X(kY)_{|Y} \simeq \mathcal L_{|Y}$.

    If $\utype(Y)>k$ then $\mathcal O_X(kY)_{|(k+1)Y} \simeq \mathcal L_{|(k+1)Y}$ and Item (1) of Theorem \ref{THM:closed}
    implies $\utype(Y)=\infty$.
\end{proof}

This result generalizes  \cite[Theorem 5.1]{Neeman}.

\section{Convergence of formal isomorphisms }\label{S:A}
This section is devoted to the proof of Theorem \ref{THM:A} from the Introduction.

\subsection{Formal diffeomorphisms preserving closed differential forms}
The convergence of formal diffeomorphism will be implied by the
following simple application of Artin's approximation theorem.

\begin{lemma}\label{L:converge}
    Let $\phi : \widehat{(\C^2,0)} \to \widehat{(\C^2,0)}$ be a formal biholomorphism.
    Suppose the existence of two pairs $(\alpha_1, \beta_1)$ and $(\alpha_2, \beta_2)$ of
    convergent exact meromorphic $1$-forms on $(\mathbb C^2,0)$ satisfying the following properties.
    \begin{enumerate}
        \item Both $\alpha_1 \wedge \beta_1$ and $\alpha_2 \wedge \beta_2$ are not identically zero.
        \item There exist constants $\mu, \nu \in \mathbb C^*$ such that $\phi^* \alpha_2 = \mu \alpha_1$ and $\phi^* \beta_2 = \nu \beta_1$
    \end{enumerate}
    Then $\phi$ is the restriction to $\widehat{(\C^2,0)}$ of a germ of bihomolomorphism, i.e. $\phi$ is convergent.
\end{lemma}
\begin{proof}
    Let $a_1, a_2, b_1, b_2$ be meromorphic first integrals of $\alpha_1, \alpha_2, \beta_1, \beta_2$ respectively.
    By suitably choosing the constants of integration, we may assume that $\phi^* a_2 = \mu a_1$ and $\phi^* b_2 = \nu b_1$.
    Moreover, if $0$ is not a pole of $a_i$ then we choose $a_1$ and $a_2$ such that $a_1(0)= a_2(0) = 0$. Similarly for $b_i$.

    Consider the system of equations
    \begin{equation}\label{E:system}
    \begin{cases}
        \mu \cdot a_1(x_1,y_1) -  a_2(x_2,y_2)  = 0 \, ,\\
        \nu \cdot b_1(x_1,y_1) -  b_2(x_2,y_2)  = 0 \, .\\
    \end{cases}
    \end{equation}
    Our assumptions imply that  $(x_2,y_2) = \phi(x_1,y_1)$ is a formal solution for the system (\ref{E:system}). For every $N \in \mathbb N$, Artin's approximation theorem \cite{Artin}, implies the existence of
    a convergent solution $(x_2,y_2) = \varphi_N(x_1,y_1)$ such that the Taylor expansion of $\varphi_N$  coincides with the one of $\phi$ up to  order $N$.
    For $N\geq 1$, the application $\varphi_N$ is a diffeomorphism.

    Consider the formal biholomorphism $ \psi_N = \phi\circ {\varphi_N}^{-1} : \widehat{(\C^2,0)} \to \widehat{(\C^2,0)}$.
    Note that $\psi_N$, and all its iterates, satisfy
    \begin{equation}\label{E:psi}
        \psi_N^* \alpha_2 = \alpha_2 \quad \text{ and } \quad \psi_N^* \beta_2 = \beta_2 \, .
    \end{equation}
    Since $\psi_N$ is tangent to the identity up to order $N$, there exists a formal vector field $v_N$ on $\widehat{(\C^2,0)}$ vanishing up to order at least $N$, such that $\psi_N  = \exp[1](v_N)$,
    i.e. $\psi_N$ is the flow of $v_N$ at time one.

    If $N\ge 2$  then  the formal vector field $v_N$ has no linear term and the  expansion of its formal flow $(t,x,y) \mapsto (\exp[t](v_N))(x,y)$, can be written as
    \[
        (t,x,y) \mapsto \sum_{i,j} p_{ij}(t)x^iy^j
    \]
    where $p_{ij} \in \mathbb C[t]$ are polynomials.  Therefore, the validity of  Equation (\ref{E:psi}) for all iterates of $\psi_N$ implies also the validity of Equation (\ref{E:psi}) for the formal bilomorphisms $(x,y) \mapsto (\exp[t](v))(x,y)$  when $t \in \mathbb C$ is arbitrary. Consequently $L_{v_N} \alpha_2 = L_{v_N} \beta_2 =0$, where $L_v = i_v d + di_v$ is the Lie derivative along $v$. As both $\alpha_2$ and $\beta_2$ are closed $1$-forms, we deduce that both $i_{v_N} \alpha_2 = s_N$ and $i_{v_N} \beta_2 = t_N$ belong to $\mathbb C$.
    Let $w_a$ and $w_b$ be meromorphic vector fields such that $\alpha_2(w_a)=1, \alpha_2(w_b)= 0, \beta_2(w_a)= 0,$ and $\beta_2(w_b)=1$. Since $\alpha_2 \wedge \beta_2 \neq 0$, the vector fields $w_a$ and $w_b$ are uniquely determined by these equations. Notice also that
    \[
        v_N = s_N \cdot w_a + t_N \cdot w_b  \, .
    \]
    Since the order of $v_N$ is at least $N$ we deduce that $v_N =0$  for $N$ sufficiently large. If $N\gg 0$ then $\psi_N$  is the identity and we conclude that $\phi$ coincides with the convergent biholomorphism $\varphi_N$.
\end{proof}

\subsection{Proof of Theorem  \ref{THM:A} } Let $S$ be a smooth projective surface
and let $C \subset S$ be a smooth curve with trivial normal bundle. Let $S'$ and $C'$
be another pair with the same properties. Let  $\mathscr C$ be  the completion of $S$
along $C$,  and let $\mathscr C'$ be the completion of $S'$ along $C'$. Assume
the existence of a formal biholomorphism $\widehat{\varphi} : \mathscr C \to \mathscr C'$.

If $C$ is a fiber of a fibration on $S$ then the same holds for $C'$ and the existence
of a  germ of biholomorphism between neighborhoods of $C$ and $C'$ follows from the
main result of \cite{Hirschowitz}.

Assume from now on that $C$ is not a fiber of a fibration. Theorem \ref{THM:closed} implies that
$\utype(C) =1$ and $\uclass(C) \in H^1(C, \mathcal I_C/\mathcal I_C^2) \simeq H^1(C, \mathcal O_C)$.
Moreover, there exists a cohomology class $a \in H^1(S,\mathcal O_S)$ such that $\uclass(C)\in H^1(C, \mathcal O_C)$ is given by the restriction of $a$ to $C$. Note that the class $a$ is the class of
the cocycle $\{a_{ij}\}$ appearing in the proof of Theorem \ref{THM:closed} in Equation (\ref{E:formula}). Let $\alpha \in H^0(S, \Omega^1_S)$ be a global holomorphic $1$-form such that $\overline{\alpha}$ coincides with the cohomology class $a \in H^1(S, \mathcal O_S)$. Note that the pull-back of $\alpha$ to
$C$ is non-zero. In particular, the foliation defined by $\alpha$ is generically transverse to $C$.

Let also $\omega \in H^0(S, \Omega^1_S (2C))$ be the closed rational $1$-form with $(\omega)_{\infty}=2C$
given by Theorem \ref{THM:closed}. The $\mathbb C$-vector space generated by the twisted $1$-form $\omega$ is not unique. The ambiguity, of course, comes from the inclusion
\[
    H^0(S, \Omega^1_S) \longrightarrow H^0(S, \Omega^1_S( 2C) ) \, .
\]
In order to choose canonically  a one dimensional subspace $\mathbb C \cdot \omega \subset H^0(S, \Omega^1_S( 2C))$ consider the representation
\begin{align*}
    \pi_1(S) &\longrightarrow  \mathbb C \\
    \gamma & \mapsto \int_\gamma \omega \, .
\end{align*}
Although $\omega$ has poles, it has no residues, and the representation above is unambiguously defined.
This representation defines an element of $H^1(S,\mathbb{C})$ which is not canonically determined due to the ambiguity $H^0(S,\Omega^1_S)$.
However its class in $H^1(S,\mathbb{C})/H^0(S,\Omega^1_S)$ is unique and by construction is given by the cohomology class $a\in H^1(S,\mathcal{O}_S)$.
Thus we can choose $\mathbb C \cdot \omega$  inside  the vector space $H^0(S, \Omega^1_S( 2C) )$ by imposing that the periods of $\omega$  are proportional to the periods of $\overline{\alpha}$, by Lemma \ref{L:restriction 1} this condition determines a one-dimensional subspace of $H^0(S,\Omega^1_S(2C))$.
Since the foliation defined by $\omega$ leaves the curve $C$ invariant while the foliation defined by $\alpha$ does not, the wedge product of $\alpha$ and $\omega$ does not vanish identically.

If $\alpha'$ and $\omega'$ are the analogue $1$-forms on $S'$ then  $\widehat{\varphi}^* \alpha'$ is proportional to $\alpha$ thanks to Lemma \ref{L:restriction 1}. Similarly, $\widehat{\varphi}^* \omega'$ is proportional to $\omega$. We apply Lemma \ref{L:converge} to conclude that $\widehat{\varphi}$ is the restriction to $\mathscr C$ of a germ of biholomorphism between the germ of $S$ along $C$ and the germ of $S'$ along $C'$. Furthermore, we can apply Item (\ref{I:sections}) of Corollary \ref{C:UedaProjective} to guarantee the existence of a rational map $\varphi: S \dashrightarrow S'$ such that $\varphi_{|\mathscr C} = \widehat{\varphi}$. Finally, $\varphi$ must be birational (i.e. of degree one) since otherwise the pre-image of $C$ would have two distinct irreducible components (all of them supporting divisors of zero self-intersection)  contradicting Item (\ref{I:Narasimhan}) of Corollary \ref{C:UedaProjective}. \qed

\begin{remark}
    Theorem \ref{THM:A} also holds for a smooth curve $C$ with
    torsion normal bundle of order $k$ in a projective surface $S$  and $\utype(C) \ge k$.
    The proof is  essentially the same. If $\utype(C)>k$ then Theorem \ref{THM:closed}
    implies $\utype(C)= \infty$. Consequently, $kC$ is a fiber of a fibration and the result
    follows from  \cite{Hirschowitz}. If instead $\utype(C)=k$ then the arguments used in the
    proof of Theorem \ref{THM:A} can be repeated almost word-by-word: the only difference is that the
    closed rational $1$-form $\omega$ will have in this case poles of order $k+1$ instead of poles of order
    $2$.
\end{remark}

\subsection{Divisors with trivial normal bundle}
Theorem \ref{THM:A} admits the following version for reduced divisors with trivial normal bundle.

\begin{thm}
    Let $(S,D)$ be pair where $S$ is a smooth projective surface and $D$ is a reduced divisor with trivial normal bundle, i.e. $\mathcal O_S(D)_{|D} \simeq \mathcal O_D$. Then $(S,D)$ satisfies the projective formal principle.
\end{thm}

The proof is the same as the proof of Theorem \ref{THM:A}. The only difference is that we
do not have available Theorem \ref{T:Ueda} for reduced divisors (or even for singular curves)
and we are not able to conclude the birational formal principle when $D$ does not move in a fibration.

\begin{problem}\label{pbm:Ueda for divisors}
    Establish versions of Ueda's Theorem \ref{T:Ueda} where the smooth curve $C$ is replaced by an arbitrary
    effective divisor with trivial, torsion, or unitary flat normal bundle.
\end{problem}

Partial results toward a positive solution to Problem \ref{pbm:Ueda for divisors} have been obtained by Ueda \cite{UedaSingular} and Koike \cite{Koike}.

\subsection{Beyond curves with trivial normal bundle}
It is conceivable that variants of the arguments used to prove Theorem \ref{THM:A} will
lead to a more general statement. Probably, the main obstruction to extending the argumentation
is our deficient understanding of the following question.

\begin{question}
    Let $Y \subset X$ be a smooth hypersurface  on a projective manifold $X$ with numerically trivial normal bundle and  $\utype(C)<\infty$. If the order of the normal bundle of $Y$ in $\Pic^{\tor}(Y/X)$ is finite, does there exist a foliation or a web on $X$ canonically attached to the pair $(X,Y)$ ?
\end{question}

A similar question was already raised in \cite[Question 4.6]{CLPT}.

\bibliography{references}{}
\bibliographystyle{amsplain}

\end{document}